\documentclass[reqno,12pt]{amsart}
\usepackage{graphicx}
\usepackage{amssymb}
\usepackage{subcaption}
\usepackage{times}

\newtheorem{theorem}{Theorem}[section]

\numberwithin{equation}{section}  

\begin{document}

\title[Whitney Forms in Spacetime]
      {Geometric Representations of Whitney Forms and their Generalization to Minkowski Spacetime}

\author[J. Salamon]{Joe Salamon}
\email{jsalamon@physics.ucsd.edu}
\author[J. Moody]{John Moody}
\email{jbmoody@math.ucsd.edu}
\author[M. Leok]{Melvin Leok}
\email{mleok@math.ucsd.edu}

\address{Department of Mathematics\\
         University of California San Diego\\ 
         La Jolla CA 92093}

\date{\today}


\begin{abstract}
In this work, we present two alternative yet equivalent representation formulae for Whitney forms that are valid for any choice of coordinates, and generalizes the original characterization of Whitney forms in \cite{Wh57} that requires the use of barycentric coordinates. In addition, we demonstrate that these formulae appropriately generalize the notion of Whitney forms and barycentric coordinates to Minkowski spacetime, and naturally to any other flat pseudo-Riemannian manifold. These alternate forms are related to each other through a duality between the exterior algebras on vectors and covectors. In addition, these two formulae have a geometrically intuitive interpretation which provide interesting insights into their structure. Furthermore, we obtain an explicit characterization of the Hodge dual of the space of Whitney forms on Minkowski spacetime, and this opens the door to treating the theory of classical electromagnetism in a fully covariant fashion through the combination of multisymplectic variational integration and spacetime finite-element exterior calculus (FEEC) techniques. We conclude the paper with the results from an $\mathbb{R}^{1+1}$ wave equation simulation, as a proof-of-concept of the spacetime formulation described herein.
\end{abstract}

\maketitle

{\footnotesize
\setcounter{tocdepth}{1}
\tableofcontents
}
\vspace*{-0.5cm}

\section{Introduction}
The finite element community has long made use of the notion of N{\'e}d{\'e}lec edge elements \cite{Ned80} and Raviart--Thomas elements \cite{RaTh1977} in the context of mixed finite-element methods. It was later made clear by Bossavit in \cite{Bos88} that these edge elements correspond exactly to the barycentric differential forms Whitney used in his classic work \cite{Wh57}. It is an interesting historical coincidence that Whitney's ``elementary forms", which were used as a coordinate-free tool for proving theorems on simplicial meshes, would arise in a computational technique developed decades later. In retrospect, this is not surprising, as the de Rham complex of differential forms provides the appropriate unified geometrical setting for formulating mixed finite-element methods. This connection was initially used to develop generalizations of edge elements in \cite{Hip02}, and a more expansive theory of finite-element exterior calculus was developed in \cite{AFWActa, AFW10}.

In this paper, we will generalize the notion of Whitney forms through the use of differential geometry to flat, pseudo-Riemannian manifolds, and show that our formulations are equivalent to that of Whitney's original notion of these forms in Euclidean spaces. Furthermore, these formulations will provide an explicit characterization of the Hodge dual of Whitney forms on flat, pseudo-Riemannian manifolds.

An important motivation for this work is the desire to apply finite-element exterior calculus to multisymplectic variational integration of Hamiltonian field theories, and to go beyond spacetime tensor-product meshes to spacetime simplicial meshes instead. This provides a natural setting to explore the potential benefits of combining compatible discretization techniques with geometric numerical integration techniques.

\subsection*{Exterior algebra of vector spaces}
Before proceeding, we briefly summarize the necessary definitions and basic results from the theory of exterior algebras in order to set our notation.
Let $V$ be a vector space of dimension $n$, and let $\Lambda(V)$ be the exterior algebra of $V$, with $\cdot \wedge \cdot : \Lambda(V) \times \Lambda(V) \rightarrow \Lambda(V)$  denoting the exterior product.
$\Lambda^k(V)$ will be the span of $v_1 \wedge v_2 \wedge \dots \wedge v_k$ over all $v_1, v_2, \dots , v_k \in V$ (by convention, we set $\Lambda^0(V) = \mathbb{R}$).
Let $\langle\cdot,\cdot\rangle : V \times V \rightarrow \mathbb{R}$ be a pseudo-Riemannian metric -- a non-degenerate bilinear mapping.
The flat operation ${\cdot}^{\flat}$, usually described as the lowering of indices through the metric tensor, is defined by $v^\flat\in V^{*}$ such that $v^\flat(w)=\langle v , w \rangle$.
For $\lambda \in V^*$, let $i_\lambda(v)=\lambda(v)$, where $v\in V$, denote  the contraction operation, and let $i_{\lambda}(v_1 \wedge v_2)=(i_\lambda v_1) \wedge v_2 + (-1)^{k_1} v_1 \wedge (i_\lambda v_2)$ for any $v_1 \in \Lambda^{k_1}(V)$ and $v_2 \in \Lambda^{k_2}(V)$.
The inner product defined above extends to the exterior algebra by $\langle v_1 \wedge v_2 , v_3 \rangle = \langle v_2 , i_{v_1^\flat} v_3 \rangle$.

We define a volume form as any $Vol \in \Lambda^n (V)$ with $\langle Vol , Vol \rangle = 1$.
The operation $\star: \Lambda^k \rightarrow \Lambda^{n-k}$ which takes $w \rightarrow i_{w^b} Vol$ is called the Hodge star.
The Hodge star is also characterized by the identity $w \wedge (\star v) = \langle w , v \rangle Vol$.

\subsection*{Exterior calculus of differential forms} When we consider the exterior algebra of differential forms, we can construct an exterior calculus by introducing the exterior derivative,  $d:\Lambda^k\rightarrow\Lambda^{k+1}$, which satisfies the usual properties:
\begin{enumerate}
\item $df$ is the differential of a smooth function $f\in\Lambda^0(V)$;
\item $d(dv) = 0$, for any form $v\in\Lambda^k(V)$, i.e., $d^2=0$;
\item $d(v\wedge w)=(dv)\wedge w + (-1)^k v\wedge(dw), \forall v\in\Lambda^k(V), w\in\Lambda^j(V)$.
\end{enumerate}

Lastly, the operation $\delta:\Lambda^k\rightarrow\Lambda^{k-1}$ is the codifferential, defined as the combination $\delta = (-1)^{nk+n+1}\star d \star$.
Note that the first and second Hodge stars in this operator act on different spaces.

As a final note, we will use the following definition for the Whitney forms:
\begin{equation}
{}^{j}w_{\rho}(x)=j!\sum_{i=0}^j (-1)^i \lambda_i(x) d\lambda_0\wedge d\lambda_1\dots \widehat{d\lambda_i}\dots d\lambda_j, \label{whitney_coordinate}
\end{equation}
where $\rho=[v_0,v_1,\dots v_j]$ is a standard, ordered subsimplex of $\sigma=[v_0,\dots,v_n], (j\leq n)$, $\lambda_i$ is the barycentric coordinate associated with the vertex $v_i$, and the hat indicates an omitted term within the wedge product.

In the next section, we will first consider the exterior algebra of covectors, i.e., differential forms in generalizing the Whitney forms. We will then consider the exterior algebra of vectors to construct a dual, yet equivalent, generalization.

\section{Whitney Forms: the Covector Approach}
The basic properties of Whitney forms lead to a natural and geometrically intuitive construction in terms of the coordinate vertex vectors of a given simplex. We will first showcase this representation, then turn to a discussion of its derivation. For the remainder of this section, we will use the notation ${}^{j}w_{\rho}$ to denote a Whitney $j$-form over the oriented set of $j+1$ vertices of a $j$-simplex $\rho$.
\begin{theorem}\label{covector_whitney_form}
Let $\sigma := [v_0,v_1,...v_n]$, an ordered set of vertex vectors, represent an oriented $n$-simplex on a flat $n$-dimensional manifold, and correspondingly, let $\rho :=[v_i,v_{i+1}...v_{i+j}]$ represent a subsimplex of $\sigma$, with $i+j\leq n$. Taking $\tau=\sigma\smallsetminus\rho=[v_0,...v_{i-1},v_{i+j+1},...v_n]$, the ordered complement of $\rho$ in $\sigma$, the Whitney $j$-form over $\rho$ can be written as 
\begin{equation}
{}^{j}w_{\rho} (x) = \frac{sgn(\rho\cup\tau)}{\star vol(\sigma)}\frac{j!}{n!}\Bigl(\star\bigwedge_{v_k\in\tau}{(v_k - x)}^{\flat}\Bigr),\label{vertexrep}
\end{equation}
where $vol(\sigma)$ is the volume form of $\sigma$, defined by $vol(\sigma)=\frac{1}{n!}\bigwedge_{i=1}^n(v_i-v_0)^{\flat}$, $sgn(\rho\cup\tau)$ is the sign of the permutation of the ordered vertex set $\rho\cup\tau$ relative to $\sigma$, and $x$ is the position vector. The terms in the wedge product are ordered as in $\tau$.
\end{theorem}
Note the interesting structure: the Whitney form, a metric-independent object, depends on the \emph{Hodge dual of its complementary simplex}, up to metric-dependent details. In fact, ${}^{j}w_{\rho}(x)=0,\forall x\in\tau$. We now prove the validity of this formula through the use of a vector proxy representation of the Whitney forms. A more direct (but possibly less intuitive) proof is presented in the appendix.

\section{Whitney Forms: the Vector Approach}

Whitney forms can also be obtained by working directly with the exterior algebra on the vector space, and using the inner product to extend the metric to this graded algebra.

Let $\sigma$ be an $n$-simplex, described by an ordered set of vertices $[v_0,v_1,\ldots, v_n]$, with volume multivector $Vol(\sigma)=(vol(\sigma))^{\sharp}=\frac{1}{n!}\bigwedge_{i=1}^{n}{(v_i - v_0)}$. Define the multivector
\begin{equation*}
V_{\rho} = \bigwedge\limits_{v_i\in\rho}(v_i-x),
\end{equation*}
where $x$ is the position vector, and $\rho$ is an ordered subsimplex of $\sigma$ and the terms in the wedge product are to be taken in the same order as $\rho$. If $\rho$ is taken to be a $j$-subsimplex of $\sigma$, then $V_{\rho}$ represents the geometric $(j+1)$-volume multivector of $\rho$ and $x$, with $x$ taken as the origin. Then, we have:

\begin{theorem}\label{main_theor}
Let $x$ be a point in the $n$-simplex. Given a $j$-multivector $U=\bigwedge_{i=1}^{j}u_i$, the evaluation of Whitney $j$-form ${}^{j}w_{\rho}$ on $U$ is given by
\begin{equation}
{}^{j}w_{\rho}(x)[U]=sgn(\rho\cup\tau)\frac{j!}{n!}\frac{\left\langle Vol(\sigma),V_{\tau}\wedge U\right\rangle}{\langle Vol(\sigma), Vol(\sigma)\rangle},
\end{equation}
where $\tau=\sigma\smallsetminus\rho$.
\end{theorem}
\begin{proof}
When $x\in\tau$, both ${}^{j}w_{\rho}(x)[U]$ and the ratio of inner products vanish identically.

Now, we consider $x\in\rho$, and express it in terms of barycentric coordinates, $x=\sum {\lambda_k v_k}$. Then, by multilinearity, it suffices to establish the validity of the expression for $x=v_k$, where $v_k$ is a vertex of $\rho$. We must ensure that the two sides of the equation agree on the basis of $j$-multivectors. Thus, let $\alpha$ be a $(j-1)$-subsimplex with $j$ vertices that does not contain $v_k$, then take $U$ to be the $j$-volume multivector of $\alpha\cup v_k$, where the terms are in increasing order. More explicitly, take the following set as a basis for $U$:
\begin{equation*}
\left\{\bigwedge_{i=1}^{j}(v_{\sigma(i)}-v_k)|\sigma\in Y\right\},
\end{equation*}
where $Y$ is the set of all increasing functions that map $\{1,\ldots,j\}\rightarrow\{0,\ldots,\hat{k},\ldots,n\}$. If the image $v_{\sigma}\ne\rho$, then $v_{\sigma}\cup\tau\ne\emptyset$ and thus $V_{\tau}\wedge U = 0$. However, if $v_{\sigma}=\rho$, then
\begin{equation*}
sgn(\rho\cup\tau)\frac{j!}{n!}\frac{\left\langle Vol(\sigma),V_{\tau}\wedge U\right\rangle}{\langle Vol(\sigma), Vol(\sigma)\rangle}=sgn(\rho\cup\tau)j!\frac{\left\langle Vol(\sigma),sgn(\rho\cup\tau) Vol(\sigma)\right\rangle}{\langle Vol(\sigma), Vol(\sigma)\rangle}=j!
\end{equation*}
which agrees with ${}^{j}w_{\rho}(v_k)[U]=j!$
\end{proof}
Thus, \ref{main_theor} is the appropriate vector proxy of the barycentric representation of Whitney forms in any flat space.

\section{Representational Connections}
The vector formulation presented above leads naturally to the covector formulation showcased in \eqref{vertexrep}. In fact, we claim the ratio of inner products in \eqref{main_theor} is exactly equivalent to that differential form contracted over the $j$-vector $U=\bigwedge_{i=1}^{j}u_i$:
\[
{}^{j}w_{\rho}(x)[U]=sgn(\rho\cup\tau)\frac{j!}{n!}\frac{\left\langle Vol(\sigma),V_{\tau}\wedge U\right\rangle}{\langle Vol(\sigma), Vol(\sigma)\rangle}=\frac{sgn(\rho\cup\tau)}{\star vol(\sigma)}\frac{j!}{n!}\Bigl(\star\bigwedge_{v_k\in\tau}{(v_k - x)}^{\flat}\Bigr)[U]
\]
We now turn to a direct proof of this equation and Theorem \ref{covector_whitney_form}.
\begin{proof}
\begin{align*}
{}^{j}w_{\rho}(x)[U]&=sgn(\rho\cup\tau)\frac{j!}{n!}\frac{\left\langle Vol(\sigma),V_{\tau}\wedge U\right\rangle}{\langle Vol(\sigma), Vol(\sigma)\rangle}=sgn(\rho\cup\tau)\frac{j!}{n!}\frac{\left\langle vol(\sigma)^{\sharp},V_{\tau}\wedge U\right\rangle}{\left\langle vol(\sigma)^{\sharp},vol(\sigma)^{\sharp}\right\rangle}
\\
&=sgn(\rho\cup\tau)\frac{j!}{n!}\frac{\left\langle vol(\sigma)^{\sharp},V_{\tau}\wedge U\right\rangle}{\left\langle \star vol(\sigma),\star vol(\sigma)\right\rangle}=sgn(\rho\cup\tau)\frac{j!}{n!}\frac{\left\langle vol(\sigma),V_{\tau}^{\flat}\wedge U^{\flat}\right\rangle}{\left\langle \star vol(\sigma),\star vol(\sigma)\right\rangle}
\\
&=\frac{sgn(\rho\cup\tau)}{(\star vol(\sigma))^2}\frac{j!}{n!}\left\langle i_{V_{\tau}}(vol(\sigma)),U^{\flat}\right\rangle=\frac{sgn(\rho\cup\tau)}{\star vol(\sigma)}\frac{j!}{n!}\left\langle i_{V_{\tau}}Vol,U^{\flat}\right\rangle
\\
&=\frac{sgn(\rho\cup\tau)}{\star vol(\sigma)}\frac{j!}{n!}(i_{V_{\tau}}Vol)[U]=\frac{sgn(\rho\cup\tau)}{\star vol(\sigma)}\frac{j!}{n!}(\star (V_{\tau}^{\flat}))[U]
\\
&=\frac{sgn(\rho\cup\tau)}{\star vol(\sigma)}\frac{j!}{n!}\Bigl(\star\bigwedge_{v_k\in\tau}{(v_k - x)}^{\flat}\Bigr)[U]
\end{align*}
where $Vol$ represents the volume element of the space.
\end{proof}

Thus, the covector formulation represents the appropriate generalization of Whitney forms to flat, pseudo-Riemannian spaces.

\section{Explicit Examples of Generalized Whitney Forms}
In this section, we will demonstrate what these generalized Whitney forms look like in typical flat, pseudo-Riemannian spaces in comparison to their Riemannian counterparts.
\begin{itemize}
\item Consider a simplex $\sigma=[v_0,v_1,v_2]$ In Euclidean $n=2$ space. The Whitney $1$-form over $\rho=[v_0,v_1]$ is given by:
\[
{}^{1}w_{[v_0,v_1]} = \frac{sgn(v_0,v_1,v_2)}{\star vol(\sigma)}\frac{1!}{2!}(\star((v_2-x)^{\flat})) = \frac{\star(v_2-x)^{\flat}}{2 \star vol(\sigma)}.
\]
Figure \ref{fig:Euclidean2} portrays $({}^{1}w_{[v_0,v_1]})^{\sharp}$, or ${}^{1}w_{[v_0,v_1]}$ raised to a vector field.
\item If we modify the previous example to a Lorentzian $n=2$ manifold with signature $(-1,1)$, ${}^{1}w_{[v_0,v_1]}$ becomes:
\[
{}^{1}w_{[v_0,v_1]} = \frac{\star(v_2-x)^{\flat}}{2 \star vol(\sigma)},
\]
as only the definition of the Hodge star changes between the two spaces. Figure \ref{fig:Lorentzian1+1} portrays $({}^{1}w_{[v_0,v_1]})^{\sharp}$. Note that as a form, ${}^{1}w_{[v_0,v_1]}$ looks identical in both $\mathbb{R}^2$ and $\mathbb{R}^{1+1}$, as it always measures the circulation around the subsimplex.
\begin{figure}
\centering
\subcaptionbox{The Whitney form on $\mathbb{R}^2$ essentially captures the circulation around $\rho$.\label{fig:Euclidean2}}{\includegraphics[width=0.48\textwidth]{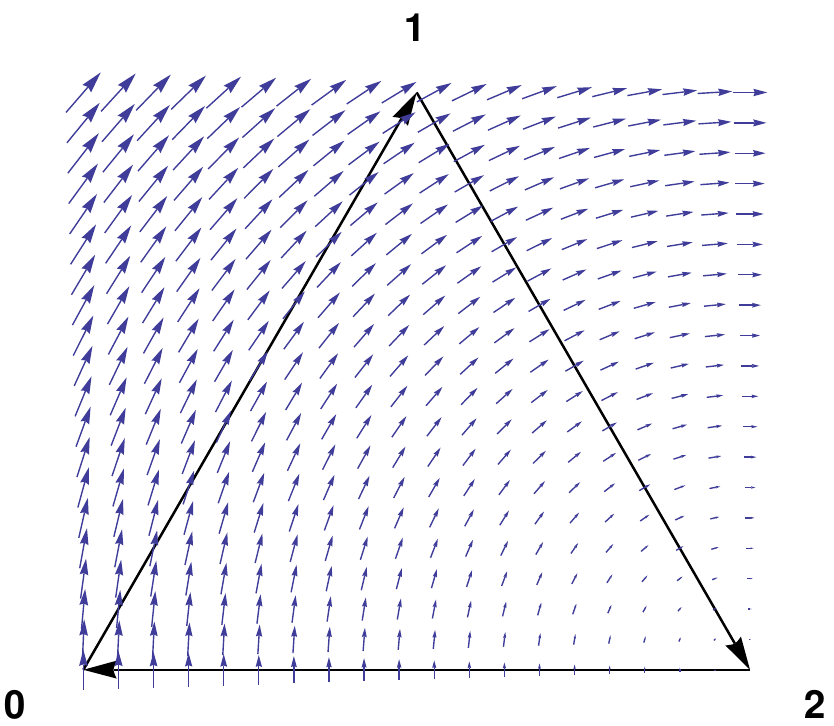}}\quad
\subcaptionbox{The Whitney form on $\mathbb{R}^{1+1}$ measures hyperbolic ``circulation", similar to a shearing of axes.\label{fig:Lorentzian1+1}}{\includegraphics[width=0.48\textwidth]{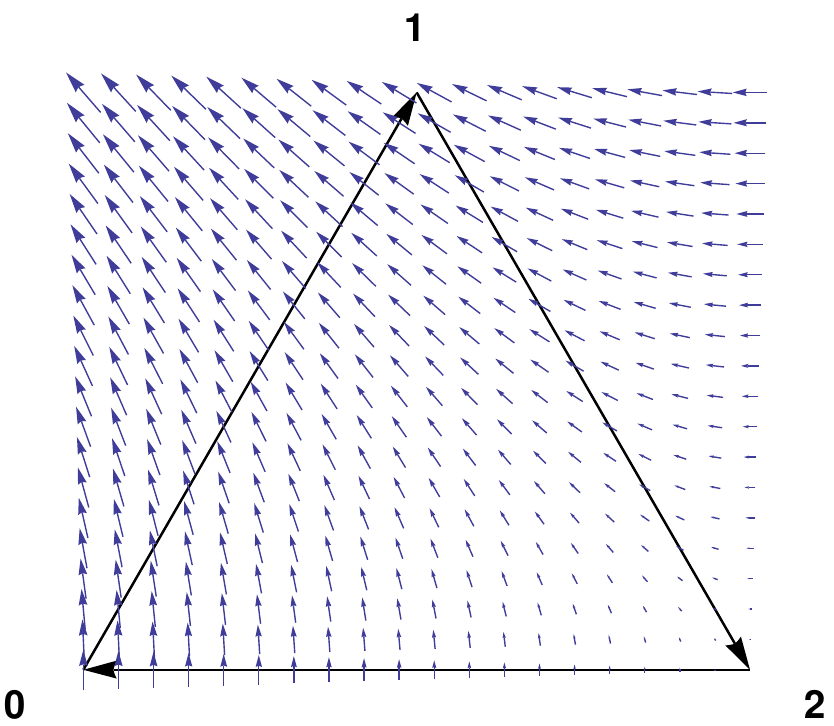}}
\caption{Visualization of the vector proxy field for the Whitney 1-form $1$-form ${}^1w_{[v_0,v_1]}$ on the $2$-simplex $[v_0,v_1,v_2]$ on a $2$-dimensional Euclidean manifold and a $(1+1)$-dimensional Lorentzian manifold. These images were produced with \textit{Mathematica}.}
\end{figure}
\item Take our manifold to be Euclidean $n=3$ space with $\sigma=[v_0,...v_3]$, we can write the Whitney $2$-form over $\rho=[v_2,v_1,v_3]$ as
\[
{}^{2}w_{[v_2,v_1,v_3]} = \frac{sgn(v_2,v_1,v_3,v_0)}{\star vol(\sigma)}\frac{2!}{3!}(\star((v_0-x)^{\flat})) = \frac{\star(v_0-x)^{\flat}}{3\star vol(\sigma)},
\]
which, upon converting to the standard vector notation, becomes
\[
{}^{2}\vec{w}_{[v_2,v_1,v_3]} = \frac{\vec{v_0}-\vec{x}}{3 |vol(\sigma)|}.
\]
In other words, the Whitney $2$-forms in $3$-space can be visualized as radial vector fields with the ``opposite" vertex as its source.
\item Now consider $n=4$ Minkowski spacetime with metric signature $(+---)$, with $\sigma=[v_0,...v_4]$ as the simplex of interest. Then, the Whitney $2$-form over $\rho=[v_1,v_0,v_4]$ is given by
\begin{align*}
{}^{2}w_{[v_1,v_0,v_4]} &= \frac{sgn(v_1,v_0,v_4,v_2,v_3)}{\star vol(\sigma)}\frac{2!}{4!}(\star((v_2-x)^{\flat}\wedge(v_3-x)^{\flat}))\\
& = -\frac{\star((v_2-x)^{\flat}\wedge(v_3-x)^{\flat})}{12 \star vol(\sigma)}.
\end{align*}
The Whitney $1$-form over $\rho=[v_3,v_1]$ in this same space is given by
\[
{}^{1}w_{[v_3,v_1]} = -\frac{\star((v_4-x)^{\flat}\wedge(v_2-x)^{\flat}\wedge(v_0-x)^{\flat})}{24 \star vol(\sigma)}.
\]
\end{itemize}

\section{Consequences}
We will now explore a few consequences of \eqref{vertexrep}. Expanding the wedge product in the formula, we obtain:
\begin{equation}
{}^{j}w_{\rho} = \frac{sgn(\rho\cup\tau)}{\star vol(\sigma)}\frac{j!}{n!}\star\Bigl(\bigwedge_{v_k\in\tau}{v_k}^{\flat}-x^{\flat}\wedge\sum_{v_k\in\tau}(-1)^{\alpha_{k}}\bigwedge_{\substack{v_l\in\tau \\ l\ne k}}{v_l}^{\flat}\Bigr),\label{linearwhitneywedgexpansion}
\end{equation}
where $\alpha_{k}$ is the number of transpositions required to bring $x^{\flat}$ to the front of the wedge product. If we take $\rho=[v_0,v_1,...v_j]$ as a standard vertex ordering for the subsimplex of interest, the above expression expands to:
\begin{equation}
\frac{1}{\star vol(\sigma)}\frac{j!}{n!}\star\Bigl(v_{j+1}^{\flat}\wedge...\wedge v_n^{\flat}-x^{\flat}\wedge\sum_{k=j+1}^{n}(-1)^{k-j-1}v_{j+1}^{\flat}\wedge v_{j+2}^{\flat}\wedge...\hat{v}^{\flat}_k\wedge...v_n^{\flat}\Bigr), \label{wedgexpansion}
\end{equation}
where, as usual, the hat indicates that the term is omitted. With respect to some origin, the first term is akin to the volume form of the complementary simplex $\tau$, and the second term represents the sum of volume forms enclosed by the position $1$-form $x^{\flat}$ and each of $\tau$'s subsimplices. Whitney forms can thus be thought of as the Hodge dual to the difference between these two volume forms, up to a scaling factor. This is easier to picture if one realizes that the product $\bigwedge_{v_k\in\tau}{(v_k - x)}^{\flat}$ is the $(n-j)$-volume form of the simplex formed by $\tau$'s vertices with the position $x$ as the origin. 

Additionally, this representation leads us to an interesting conclusion: the Hodge dual of a Whitney form is not a Whitney form on a simplex. Indeed, the dual Whitney forms only form an orthonormal basis on the space of the dual simplex $\star\sigma$ and its dual subsimplices:
\[
\int_{\rho}{}^{j}w_{\rho}=\int_{\star\rho}\star({}^{j}w_{\rho})=1,
\]
where the dual simplex is defined in terms of the geometric Hodge dual introduced in \cite{Har06}. This implies that the Hodge dual Whitney forms also form a basis, but only on the geometric Hodge dual of the original simplicial complex. Oddly enough, the formula for the Hodge dual Whitney form is slightly easier to interpret in this representation:
\begin{equation}
\star{}^{j}w_{\rho}= (\star\star)\frac{sgn(\rho\cup\tau)}{\star vol(\sigma)}\frac{j!}{n!}\Bigl(\bigwedge_{v_m\in\tau}{(v_m - x)}^{\flat}\Bigr),
\end{equation}
since $\star\star$ is the identity map, up to a sign that depends on the index of the metric, the dimension of the manifold, and the degree of the differential form.

We can thus picture the Hodge dual Whitney forms to be a difference of the complement's volume form from the sum of the complement's subsimplicial volume forms. Indeed, this formula is also more amenable to algebraic manipulation due to the Hodge star's partial cancellation on the RHS of the equation. For example, taking the exterior derivative of a dual Whitney form yields:
\[
d(\star{}^{j}w_{\rho})=0,
\]
as $d(x^{\flat})=0$. This implies that the $\delta {}^{j}w_{\rho}=0$ as well, where $\delta$ is the codifferential.

\begin{figure}[bp]
\centering
\subcaptionbox{Dispersion sets in quickly, visible even in the first period.\label{fig:FlatFull}}{\includegraphics[width=0.48\textwidth]{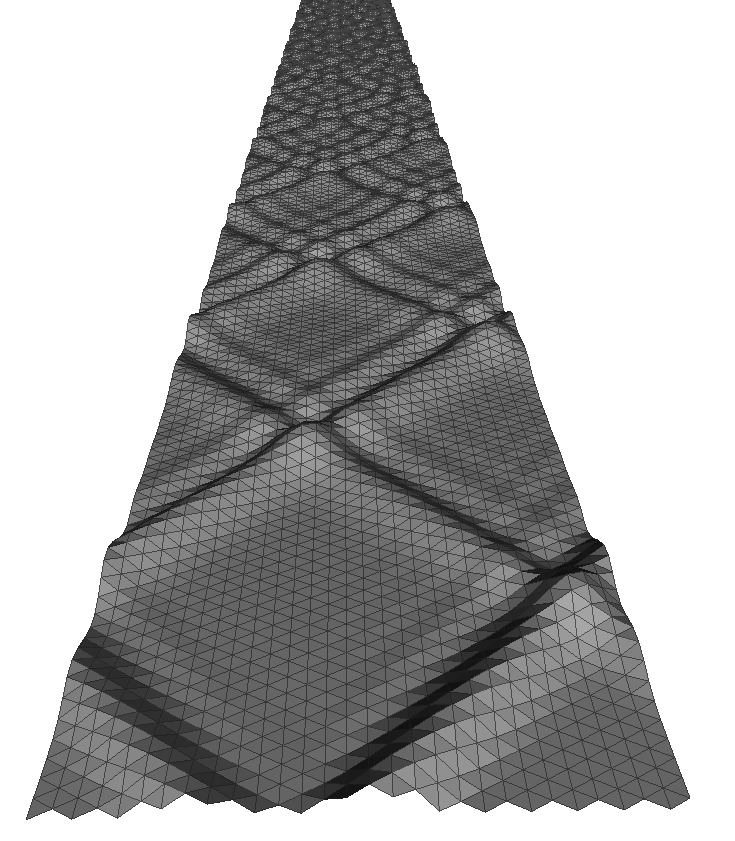}}\quad
\subcaptionbox{Towards the end of the simulation, the dispersion dominates the solution.\label{fig:FlatEnd}}{\includegraphics[width=0.48\textwidth]{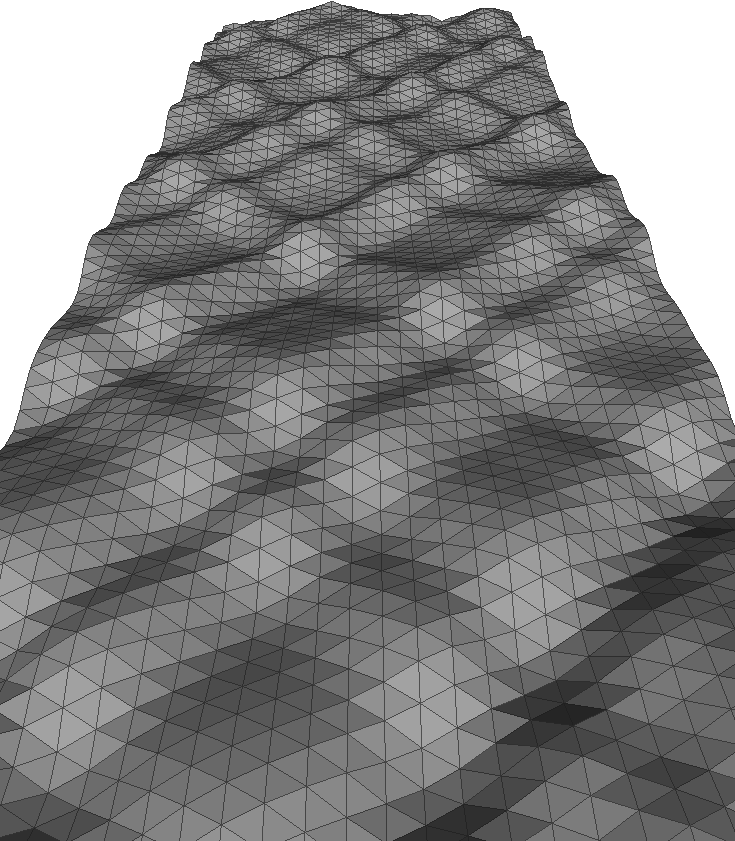}}
\caption{\label{fig:Flat}The $\mathbb{R}^{1+1}$ wave equation with a spatial, periodic boundary condition in the horizontal direction at low resolution (30 nodes per space-like slice).}
\end{figure}

\section{Application to $\mathbb{R}^{1+1}$ Wave Equation}
The coordinate-independent expressions for Whitney forms that were described above allows us to construct numerical approximations of the wave equation. More than being coordinate-independent, the algorithm is coordinate-free: it uses an abstract spacetime simplicial mesh endowed with Lorentzian distance information on the edges. In particular, the mesh on spacetime is not a tensor-product mesh, and the method does not rely on either a local or global embedding of the simplicial mesh. 

The method relies on a discretization of a multisymplectic characterization of the wave equation using a multisymplectic variational integrator~\cite{LeMaOrWe2003, MaPeShWe2001}. The variational characterization of the wave equation is in terms of the extrema of the following action functional,
\begin{align*}
S(A) &= \int_{A}{df \wedge \star df}  = \int_{A}{\langle df , df \rangle vol}.
\end{align*}
The advantage of discretizing the variational principle directly, as opposed to the equations of motion, is that the resulting numerical method is multisymplectic, and if the discretization respects the symmetries of the wave equation, then it preserves an associated multimomentum map via the discrete Noether's theorem. The discretization of the configuration bundle (fields over spacetime) using the spacetime Whitney forms that we introduced in this paper is particularly natural, since they are Lorentz-equivariant.

\begin{figure}[p]
\centering
\includegraphics[width=0.95\textwidth]{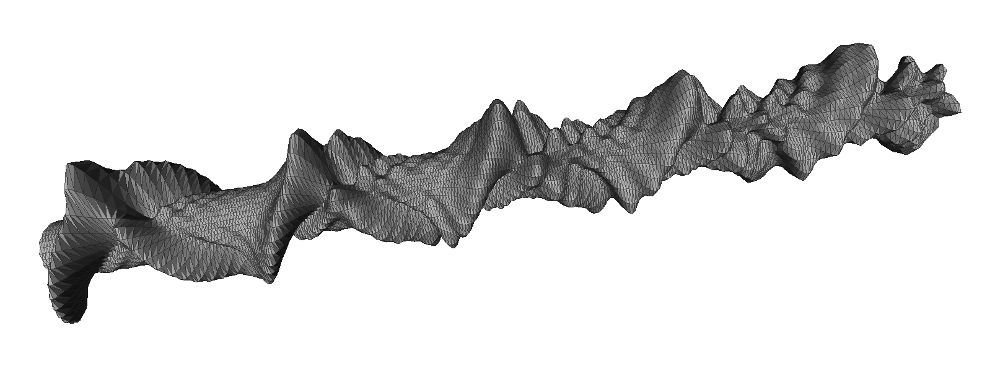}
\caption{\label{fig:Cyl1}Cylindrical visualization of the same data as in Figure~\ref{fig:FlatFull}. The cylinder's axis is the time direction; taking slices perpendicular to the time direction yields snapshots of the wave on the circle. The numerical results appear heavily distorted due to dispersion after just 2 periods.}
\end{figure}

\begin{figure}[p]
\centering
\subcaptionbox{Increasing the spatio-temporal resolution reduces dispersion.\label{fig:Cyl2}}{\includegraphics[width=0.48\textwidth]{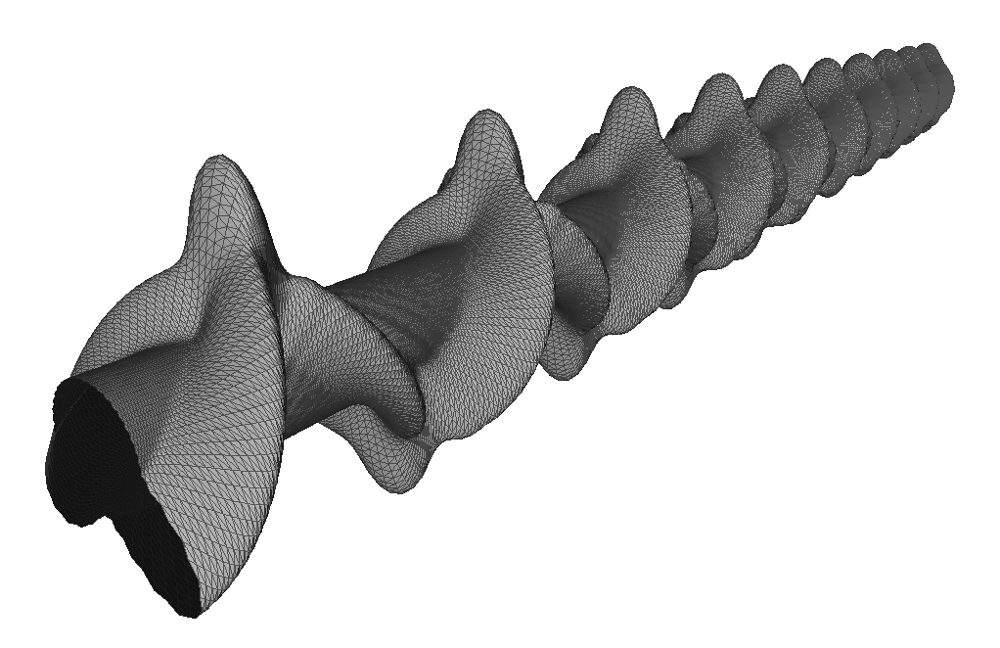}}\quad
\subcaptionbox{However, dispersion is still evident at later times (periods 8-10 are visible).\label{fig:Cyl3}}{\includegraphics[width=0.48\textwidth]{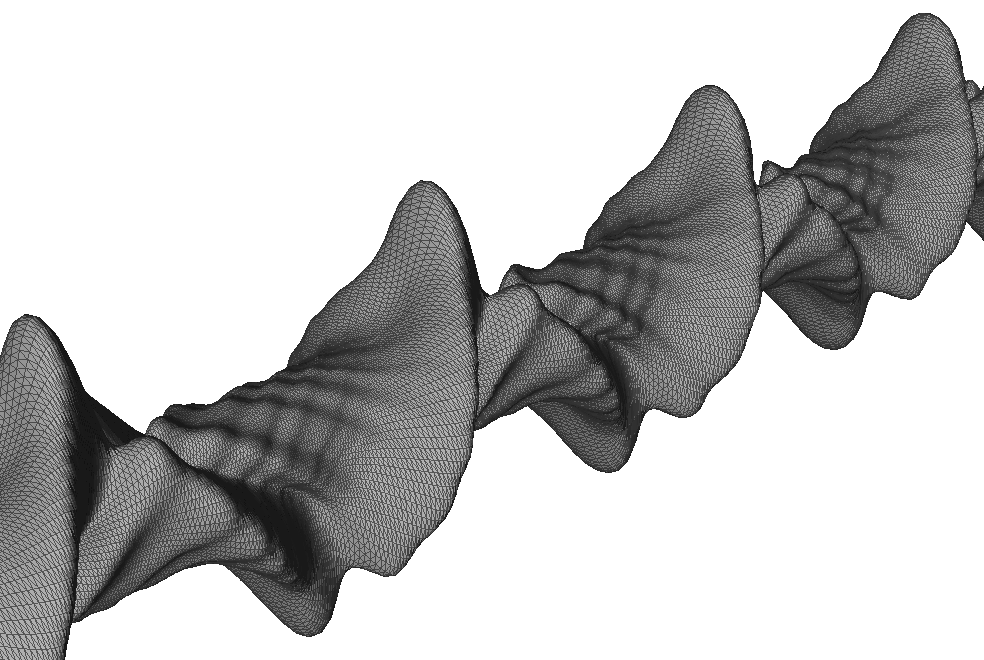}}
\caption{\label{fig:Cyl_high}Higher spatio-temporal resolution (80 nodes per space-like slice) reduces the effect of numerical dispersion in the simulation of the wave equation with the same initial data.}
\end{figure}

\begin{figure}[p]
\centering
\includegraphics[width=0.8\textwidth]{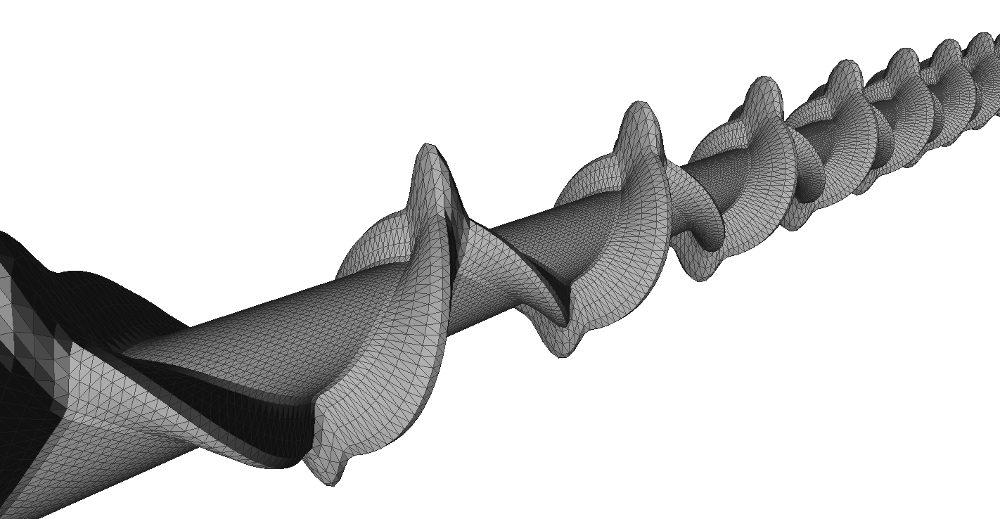}
\caption{\label{fig:Cyl4}An adapted spacetime mesh that is aligned along the integral lines, i.e., the light-cone, yields the exact solution, even at low resolution (40 nodes per space-like slice).}
\end{figure}

Instead of imposing periodic boundary conditions on a planar domain, the spacetime mesh is set up as an abstract simplicial complex that is topologically a cylinder. We present a simulation of the wave equation at low resolution in Figure~\ref{fig:Flat}, and note that the simulation exhibits numerical dispersion, but no apparent numerical dissipation. An alternative and natural cylindrical visualization of the simulation on cylindrical spacetime is given in Figure~\ref{fig:Cyl1}.

The numerical dispersion is noticeably reduced when the spatio-temporal resolution of the simulation is increased, as seen in Figure~\ref{fig:Cyl_high}. To give some indication of how the simulation performs if the mesh is well-adapted to the solution, we consider a low-resolution simulation with a spacetime mesh that is aligned with the light cones. As can be seen in Figure~\ref{fig:Cyl4}, this results in a perfect recovery of the solution. This suggests that the use of adaptive spacetime meshing may significantly improve the quality of the numerical results.

Admittedly, use of this technique is not necessary in this situation. A flat, cylindrical spacetime has a natural choice for a space-time foliation, which is exploited to produce the visualization. However, the figures presented here are a demonstration of the viability and applicability of the above generalization of Whitney forms to flat spacetimes. All the $3$-dimensional visualizations were produced with \textit{Meshlab}.

\section{Conclusion}
\label{sec:conc}

In this article we developed an appropriate generalization of Whitney forms to flat pseudo-Riemannian manifolds through the use of an alternative representation that is independent of, but still related to, the notion of barycentric coordinates. Furthermore, we demonstrated that this new formulation is completely equivalent to the barycentric Whitney forms. A consequence of the result is that we explicitly characterize the Hodge dual of the space of Whitney forms, and also obtain an alternative characterization in terms of vector proxies.

This new representation provides a means of addressing the discretization of dynamic partial differential equations from an intrinsically spacetime perspective, instead of using a spatial semi-discretization. In particular, it allows us to extend the techniques of finite-element exterior calculus (FEEC) to Minkowski spacetime, unlocking the possibility of exploiting the full relativistic symmetry of classical electromagnetism for FEEC based numerical simulations.

In addition, this representation raises the following questions:
\begin{enumerate}
\item what is the role of the dual mesh in relation to the physical variables?
\item what role does the codifferential play for the Hodge dual Whitney forms?
\item what is the appropriate generalization to curved manifolds?
\end{enumerate}
We aim to address these questions in future work. In addition, we will explore the application of the spacetime Whitney forms introduced in this paper to the spacetime discretization of classical electromagnetism using the multisymplectic variational integrator construction of \cite{LeMaOrWe2003}, with a particular emphasis on the appropriate discretization of the gauge symmetry group for electromagnetism. We also intend to combine the methods developed in this paper with the approach in \cite{RaBo2009}, in order to construct higher-order spacetime Whitney forms.

\section{Acknowledgments}
   \label{sec:ack}

This work was supported in part by NSF Grants CMMI-1029445, DMS-1065972, CMMI-1334759, and NSF CAREER Award DMS-1010687.

\section{Appendix: Direct Proof of the Covector Formulation}
In this section, we will present the more direct (but less intuitive) proof of equivalence of \eqref{vertexrep} to the standard barycentric formulation of Whitney forms, showing that it is the appropriate generalization to flat, pseudo-Riemannian spaces.
\begin{proof}
We will consider a flat $n$-dimensional manifold endowed with a pseudo-Riemannian metric with a fixed signature that contains a $n$-simplex $\sigma$. Then, from the expression \eqref{whitney_coordinate} for a Whitney $j$-form over a $j$-subsimplex $\rho$, and the fact that $\lambda_j(v_i)=\delta_{ij}$, it follows that:
\[
{}^{j}w_{\rho}(v_i)=0,\qquad \text{for all } v_i \in \tau,
\]
where $\tau=\sigma\smallsetminus\rho$. Note that by linearity of the barycentric coordinate chart, this implies that ${}^{j}w_{\rho}(x)=0$, for all $x$ spanned by $\tau$. The linearity of the barycentric coordinate chart also implies that Whitney forms are linear in the position vector $x$. Together, these two conditions imply that Whitney forms must be composed of the wedge product of difference $1$-forms $(v_i-x)^{\flat}$, $v_i\in\tau$. Furthermore, Whitney forms are antisymmetric under vertex exchange. Thus, we are lead to a formula of the form:
\begin{equation}
{}^{j}w_{\rho} = C_{\sigma,j}sgn(\rho\cup\tau)\Biggl(\star\bigwedge_{v_k\in\tau}{(v_k - x)}^{\flat}\Biggr).\label{linearwhitney}
\end{equation}
The $\star$ is the usual Hodge star, and it is required to convert the $(k-j)$-form given by the wedge product into a $j$-form, and $sgn(\rho\cup\tau)$ is required to maintain consistency with the arbitrary ordering of the vertices in $\rho$ and $\tau$. The constant $C_{\sigma,j}$ is used to satisfy the normalization condition:
\begin{equation}
\int_{\rho}{}^{j}w_{\rho} = 1.
\end{equation}
In particular, we will show by induction on $j$, the degree of the differential form, that $C_{\sigma,j}=\frac{1}{\star vol(\sigma)}\frac{j!}{n!}$. Starting with the first base case $j=0$ and $\rho=[v_0]$, we find that
\begin{align*}
\int_{[v_0]}{}^{0}w_{[v_0]} &= {}^{0}w_{[v_0]}(v_0) = \frac{sgn(\sigma)}{\star vol(\sigma)}\frac{0!}{n!}\Biggl(\star\bigwedge_{m=1}^{n}{(v_m - v_0)}^{\flat}\Biggr)
\\
&= \frac{1}{n! \star vol(\sigma)}\star((v_1-v_0)^{\flat}\wedge...\wedge(v_n-v_0)^{\flat})=1.
\end{align*}
Thus, equation \eqref{vertexrep} for $j=0$ corresponds exactly to barycentric coordinates. Next, we consider the second base case $j=1$. Taking $\rho=[v_0,v_1]$, the LHS of the normalization condition becomes:
\begin{align*}
&\int_{[v_0,v_1]}{}^{1}w_{[v_0,v_1]}\\
&\qquad= \frac{1}{n!\star vol(\sigma)}\int_{[v_0,v_1]}\Biggl(\star\bigwedge_{m=2}^{n}{(v_m - x)}^{\flat}\Biggr)
\\
&\qquad= \frac{1}{n!\star vol(\sigma)}\int_{[v_0,v_1]}\star\biggl(v_{2}^{\flat}\wedge...\wedge v_n^{\flat}-x^{\flat}\wedge\sum_{m=2}^{n}(-1)^{n-m}v_2^{\flat}\wedge v_3^{\flat}\wedge...\widehat{{v}^{\flat}}_m\wedge...v_n^{\flat})\biggr).
\intertext{Now, we exploit the translation invariance of the representation and set $v_2$ as our origin. Note that translation invariance is not required to complete the proof, but leads to the most geometrically simple result. The expression above then simplifies to:}
&\qquad=  \frac{-1}{n!\star vol(\sigma)}\int_{[v_0-v_2,v_1-v_2]}\star((x-v_2)^{\flat}\wedge(v_3-v_2)^{\flat}\wedge...(v_n-v_2)^{\flat})),
\intertext{as only the term without $v_2^{\flat}$ survives. Since the integration is over the $1$-simplex $[v_0-v_2,v_1-v_2]$, we can parametrize our path by $x(t) = (v_1-v_0)t+(v_0-v_2)$ with $x'(t) = (v_1 - v_0)$, where $t\in[0,1]$. The expression above then becomes:}
&\qquad = \frac{-1}{n!\star vol(\sigma)}\int_{0}^{1}\biggl\langle\star(((v_1-v_0)t+(v_0-v_2))^{\flat}\wedge(v_3-v_2)^{\flat}\wedge...(v_n-v_2)^{\flat}),(v_1-v_0)^{\flat}\biggr\rangle dt 
\\
&\qquad =\frac{-1}{n!\star vol(\sigma)}\left\langle\star\biggl(\biggl(\frac{1}{2}(v_1-v_0)+(v_0-v_2)\biggr)^{\flat}\wedge(v_3-v_2)^{\flat}\wedge...(v_n-v_2)^{\flat}\biggr),(v_1-v_0)^{\flat}\right\rangle
\\
&\qquad = \frac{-1}{n!\star vol(\sigma)}\star\biggl((v_1-v_0)^{\flat}\wedge\biggl(\frac{1}{2}(v_1-v_0)+(v_0-v_2)\biggr)^{\flat}\wedge(v_3-v_2)^{\flat}\wedge...(v_n-v_2)^{\flat}\biggr)
\\
&\qquad = \frac{1}{n!\star vol(\sigma)}\star\bigl((v_0-v_2)^{\flat}\wedge(v_1-v_2)^{\flat}\wedge(v_3-v_2)^{\flat}\wedge...(v_n-v_2)^{\flat}\bigr) = 1.
\end{align*}
Therefore, equation \eqref{vertexrep} for $j=1$ corresponds exactly to Whitney $1$-forms.

Next, as our inductive hypothesis, we will assume that the barycentric Whitney $l$-forms correspond to the $j=l$ case as shown in \eqref{vertexrep}.

Then, for the Whitney $(l+1)$-forms over $\rho=[v_0,...v_{l+1}]$, we can use the following well-known decomposition from the usual formulation in barycentric coordinates:
\begin{multline*}
{}^{l+1}w_{[v_0,v_1,...v_l,v_{l+1}]} = (l+1)\frac{{}^{l}w_{[v_0,v_1,...v_l]}\wedge{}^1w_{[v_l,v_{l+1}]}}{{}^0w_{[v_l]}}
\\
=\frac{1}{\star vol(\sigma)}\frac{(l+1)!}{n!}\frac{\bigl(\star\bigwedge_{i=l+1}^n(v_i-x)^{\flat}\bigr)\wedge\Bigl(\star\bigwedge_{\substack{m=0\\m \ne l,l+1}}^n(v_m-x)^{\flat}\Bigr)}{\star\bigwedge_{\substack{k=0\\k \ne l}}^n(v_k-x)^{\flat}}.
\end{multline*}
Now, we must show that this decomposition matches what \eqref{vertexrep} yields for the LHS. Let ${}^{n-l-1}\alpha=(v_{l+2}-x)^{\flat}\wedge\dots\wedge(v_n-x)^{\flat}$, ${}^{l}\beta=(v_0-x)^{\flat}\wedge\dots\wedge(v_{l-1}-x)^{\flat}$, and ${}^{1}\gamma=(v_{l}-x)^{\flat}$. Then the RHS above becomes:
\[
\frac{1}{\star vol(\sigma)}\frac{(l+1)!}{n!}\frac{(\star(\gamma\wedge\alpha)\wedge(\star(\beta\wedge\alpha))}{\star(\beta\wedge\gamma\wedge\alpha)}.
\]
Applying the identity $i_{v^{\sharp}}u=(\star\star)\star(\star u\wedge v)$ (see Lemma 8.2.1 of \cite{Hi2003}), the fraction above turns into:
\begin{align*}
\frac{1}{\star vol(\sigma)}\frac{(l+1)!}{n!}\frac{i_{\gamma^{\sharp}}(\star\alpha)\wedge i_{\beta^{\sharp}}(\star\alpha)}{i_{(\beta\wedge\gamma)^{\sharp}}(\star\alpha)}&=\frac{1}{\star vol(\sigma)}\frac{(l+1)!}{n!}\star\alpha
\\
&=\frac{1}{\star vol(\sigma)}\frac{(l+1)!}{n!}\Bigl(\star((v_{l+2}-x)^{\flat}\wedge\dots\wedge(v_n-x)^{\flat})\Bigr)\\
&=\frac{1}{\star vol(\sigma)}\frac{(l+1)!}{n!}\Bigl(\star\bigwedge_{i=l+2}^{n}(v_i-x)^{\flat}\Bigr).
\end{align*}
\end{proof}

\bibliographystyle{abbrv}
\bibliography{jjm}

\end{document}